\documentclass{article}
\usepackage{amssymb,amsmath,amsfonts,amsthm}
\usepackage{fullpage}
\usepackage{authblk}
\usepackage{graphicx}
\usepackage{enumitem}
\usepackage{tikz}
\usepackage{lineno}
\usepackage{caption}
\usepackage{hyperref}
\usepackage{cleveref}
\usepackage{cite,color}
\hypersetup{
    colorlinks,
    linkcolor={red!60!black},
    citecolor={green!60!black},
    urlcolor={blue!60!black},
    pdftitle={Spectral integral variation of signed graphs},
    pdfauthor={Jungho Ahn, Cheolwon Heo, and Sunyo Moon}
}

\newtheorem{theorem}{Theorem}[section]
\newtheorem{lemma}[theorem]{Lemma}
\newtheorem{corollary}[theorem]{Corollary}
\newtheorem{proposition}[theorem]{Proposition}
\newtheorem{observation}[theorem]{Observation}

\newtheorem{claim}{Claim}
\newenvironment{subproof}[1][\proofname]{
	
	\begin{proof}[#1]}{\end{proof}
}
\def\eqref#1{\textcolor{red!60!black}{(\ref{#1})}}
\newcommand\abs[1]{\lvert #1\rvert}

\begin{document}
\date{\today}
\title{Spectral integral variation of signed graphs}
\author[1]{Jungho Ahn\thanks{Supported by the KIAS Individual Grant (CG095301) at Korea Institute for Advanced Study.}}
\author[1]{Cheolwon Heo\thanks{Supported by the KIAS Individual Grant (CG092601) at Korea Institute for Advanced Study.}}
\author[1]{Sunyo Moon\thanks{Supported by the KIAS Individual Grant (CG092301) at Korea Institute for Advanced Study.}}
\affil[1]{Korea Institute for Advanced Study (KIAS), Seoul, South~Korea}
\affil[ ]{\small\textit{Email addresses:}
    \texttt{junghoahn@kias.re.kr},
    \texttt{cwheo@kias.re.kr},
    \texttt{symoon@kias.re.kr}
}

\maketitle

\begin{abstract}
    We characterize when the spectral variation of the signed Laplacian matrices is integral after a new edge is added to a signed graph.
    As an application, for every fixed signed complete graph, we fully characterize the class of signed graphs to which one can recursively add new edges keeping spectral integral variation to make the signed complete graph.
\end{abstract}

\section{Introduction}

In this paper, every graph is simple and finite.
For a positive integer~$n$, we let $[n]:=\{1,\ldots,n\}$ and denote by~$K_n$ the complete graph on $n$ vertices.
The \emph{Laplacian matrix} of $G$ is $L(G):=D(G)-A(G)$ where $D(G)$ is the diagonal matrix of vertex degrees of~$G$ and $A(G)$ is the adjacency matrix of~$G$.
The \emph{signless Laplacian matrix} of~$G$ is $Q(G):=D(G)+A(G)$.
A graph~$G$ is \emph{Laplacian integral} (resp. \emph{$Q$-integral}) if every eigenvalue of~$L(G)$ (resp.~$Q(G)$) is an integer.
The Laplacian integrality and $Q$-integrality of a graph have been widely investigated.
We refer to~\cite{mohar91,grone94,merris94,fallat05,merris98,balinska02,lima07,grone08,hammer96,kirkland08,kirkland10,merris97,del-vecchio18,novanta21L,goldberger13,huang15,liu10} for results of Laplacian integrality and to~\cite{stanic07,simic08,stanic09,simic10,pokorny13,park19,novanta21Q,zhang17,zhao13,freitas09,freitas10} for results of $Q$-integrality.

The (signless) Laplacian matrix of a graph can be generalized to a signed graph.
Harary~\cite{harary53} first introduced signed graphs.
A \emph{signed graph} is a pair $(G,\Sigma)$ of a graph $G$ and a subset $\Sigma$ of $E(G)$.
We call~$G$ the \emph{underlying graph} of $(G,\Sigma)$ and $\Sigma$ the \emph{sign} of $(G,\Sigma)$.
The edges in $\Sigma$ are called \emph{odd} edges, and the other edges of $G$ are called \emph{even} edges.
For $n:=\abs{V(G)}$, the \emph{adjacency matrix} of $(G,\Sigma)$ is an $n\times n$ matrix, denoted by $A(G,\Sigma)$, such that for $i,j\in[n]$, its $(i,j)$-entry is $1$ if the $i$-th vertex and the $j$-th vertex are joined by an even edge, $-1$ if they are joined by an odd edge, and $0$ otherwise.
For a vertex~$v$ of~$G$, we denote by $N_G(v)$ the neighborhood of~$v$ in~$G$.
The \emph{odd neighborhood} (resp. \emph{even neighborhood}) of~$v$ in~$(G,\Sigma)$, denoted by $N_{G,\Sigma}^-(v)$ (resp. $N_{G,\Sigma}^+(v)$), is the set of neighbors of~$v$ joined by an odd edge (resp. even edge).
We may omit the subscripts if it is clear from the context.
The \emph{signed Laplacian matrix} of $(G,\Sigma)$ is $L(G,\Sigma):=D(G)-A(G,\Sigma)$.
Note that $L(G)=L(G,\emptyset)$ and $Q(G)=L(G,E(G))$.
A signed graph is \emph{Laplacian integral} if every eigenvalue of its signed Laplacian matrix is an integer.
The \emph{spectrum} of $(G,\Sigma)$ is the spectrum of $L(G,\Sigma)$, that is, the multiset of eigenvalues of $L(G,\Sigma)$.

Similar to (signless) Laplacian matrices, the Laplacian integrality of signed graphs has generated interests.
Hou, Li, and Pan~\cite{hou03} showed that the spectrum of the signed Laplacian matrix of a signed graph is invariant under taking a switching operation at a vertex.
For a signed graph $(G,\Sigma)$ and a vertex $v$ of $G$, \emph{switching at $v$} is an operation of replacing $\Sigma$ with the sign obtained from $\Sigma$ by converting every odd edge incident with $v$ to an even edge and vice versa.
In other words, for the set $\delta(v)$ of edges of~$G$ incident with~$v$, the switching at~$v$ is an operation of replacing $\Sigma$ with the symmetric difference $\Sigma\triangle\delta(v)$.
We remark that when we switch at more than one vertices, the order of the vertices does not affect the resulting signed graph.
Two signed graphs are \emph{switching equivalent}, or \emph{equivalent} for short, if one of them is obtained from a sequence of switchings from the other.
One can easily check that the relation is indeed an equivalence relation, and we will denote this relation by~$\sim$.
For certain classes $\mathcal{C}$ of graphs, it has been revealed the structure of Laplacian integral signed graphs whose underlying graphs are in~$\mathcal{C}$.
Wang and Hou~\cite{wang21} characterized Laplacian integral signed graphs of maximum degree at most~$3$, and Wang and Gao~\cite{wang23} characterized Laplacian integral signed graphs whose underlying graphs are obtained from trees by adding at most~$3$ edges.

While there are some structural results of Laplacian integral signed graphs, a general methodology of constructing Laplacian integral signed graphs has been not known so far.
On the other hand, for graphs, So~\cite{so99} and Kirkland~\cite{kirkland04} characterized the way of constructing a Laplacian integral graph from another Laplacian integral graph by adding a new edge.
For non-adjacent vertices~$v$ and~$w$ of a graph~$G$, we denote by $G+vw$ the graph obtained from~$G$ by adding a new edge~$vw$.
For a subset~$X$ of~$E(G)$, we denote by $G\setminus X$ the graph obtained from~$G$ by removing all edges in~$X$.
If $X=\{e\}$, then we may write $G\setminus e$ for~$G\setminus X$.
We say that \emph{spectral integral variation} occurs under the addition of the new edge~$vw$ if the spectra of $L(G)$ and $L(G+vw)$ differ by integer quantities.
By Cauchy's Interlacing theorem (see~\cite{horn13}), if spectral integral variation occurs under an edge addition, then two cases happen: either one eigenvalue of $L(G)$ is increased by $2$, or two eigenvalues of $L(G)$ are increased by $1$.
We say that spectral integral variation is of \emph{type~1} if the former case happens, and of \emph{type~2} if the latter case happens.
So~\cite{so99} characterized spectral integral variation of type~1 and Kirkland~\cite{kirkland04} characterized that of type~2.
Note that if we add an edge with spectral integral variation to a Laplacian integral graph, then the resulting graph is also Laplacian integral.

For a signed graph $(G,\Sigma)$ and non-adjacent vertices~$v$ and~$w$ of~$G$, we say that \emph{spectral integral variation} occurs under the addition of a new even (resp. odd) edge $vw$ if the spectra of $L(G,\Sigma)$ and $L(G+vw,\Sigma)$ (resp. $L(G+vw,\Sigma\cup\{vw\})$) differ by integer quantities.
As in the graph cases, by the interlacing theorem of signed graphs in~{\cite[Theorem~2.5]{belardo15}}, the same two cases of spectral integral variation happens under an edge addition to a signed graph, so we use the same notions of types for signed graphs.

In this paper, we initiate a systematic study of building Laplacian integral signed graphs by extending the characterizations of spectral integral variation in~\cite{so99,kirkland04}.
We denote by $\mathbf{1}_n$ (resp. $\mathbf{0}_n$) the $n$-dimensional vector whose coordinates are ones (resp. zeros).
We may omit the subscripts of~$\mathbf{1}_n$ and~$\mathbf{0}_n$ if it is clear from the context.
Here are our main results.

\begin{theorem}\label{thm:main1 even}
    For a signed graph $(G,\Sigma)$ and non-adjacent vertices~$v$ and~$w$ of~$G$, spectral integral variation of type~1 occurs under the addition of an even edge $vw$ if and only if $N^-(v)=N^-(w)$ and $N^+(v)=N^+(w)$.
\end{theorem}

\begin{theorem}\label{thm:main2 even}
    Let $(G,\Sigma)$ be a signed graph and~$v$ and~$w$ be non-adjacent vertices of $G$ with degrees~$d_1$ and~$d_2$, respectively.
    If $L(G,\Sigma)$ is given as 
    \begin{equation}
        \left[\begin{array}{rr|r|r|r|r|r}
            d_1 & 0 & -\mathbf{1}^t & \mathbf{0}^t & -\mathbf{1}^t & -\mathbf{1}^t & \mathbf{0}^t\\
            0 & d_2 & \mathbf{0}^t & -\mathbf{1}^t & -\mathbf{1}^t & \mathbf{1}^t & \mathbf{0}^t\\\hline
            -\mathbf{1} & \mathbf{0} & L_{11} & L_{12} & L_{13} & L_{14} & L_{15}\\\hline
            \mathbf{0} & -\mathbf{1} & L_{21} & L_{22} & L_{23} & L_{24} & L_{25}\\\hline
            -\mathbf{1} & -\mathbf{1} & L_{31} & L_{32} & L_{33} & L_{34} & L_{35}\\\hline 
            -\mathbf{1} & \mathbf{1} & L_{41} & L_{42} & L_{43} & L_{44} & L_{45}\\\hline
            \mathbf{0} & \mathbf{0} & L_{51} & L_{52} & L_{53} & L_{54} & L_{55}
        \end{array}\right]
        \label{eq:centered}
    \end{equation}
    where the first and the second columns are indexed by~$v$ and~$w$, respectively, then spectral integral variation of type~2 occurs under the addition of an even edge $vw$ if and only if the following conditions hold:
    \begin{align}
        L_{11}\mathbf{1}-L_{12}\mathbf{1}+2L_{14}\mathbf{1}&=(d_2+1)\mathbf{1},\label{eq:main 1}\tag{a}\\
        L_{21}\mathbf{1}-L_{22}\mathbf{1}+2L_{24}\mathbf{1}&=-(d_1+1)\mathbf{1},\label{eq:main 2}\tag{b}\\
        L_{31}\mathbf{1}-L_{32}\mathbf{1}+2L_{34}\mathbf{1}&=-(d_1-d_2)\mathbf{1},\label{eq:main 3}\tag{c}\\
        L_{41}\mathbf{1}-L_{42}\mathbf{1}+2L_{44}\mathbf{1}&=(d_1+d_2+2)\mathbf{1},\label{eq:main 4}\tag{d}\\
        L_{51}\mathbf{1}-L_{52}\mathbf{1}+2L_{54}\mathbf{1}&=\mathbf{0}.\label{eq:main 5}\tag{e} 
    \end{align}
\end{theorem}

In~\Cref{sec:siv}, we discuss how to obtain an equivalent signed graph whose signed Laplacian matrix is as in~\eqref{eq:centered}, thus~\Cref{thm:main2 even} actually covers every spectral integral variation of type~2 occurred by adding an even edge.
We also present analogues of these theorems for the cases of adding odd edges in the section.
We remark that the proof of~\Cref{thm:main2 even} leads to a simpler proof for the characterization in~\cite{kirkland04}.

Kirkland~\cite{kirkland05} introduced integrally completable graphs.
A graph $G$ is \emph{integrally completable} if there exists a sequence $G_0,\ldots,G_m$ of Laplacian integral graphs such that $G_0=G$, $G_m$ is a complete graph, and for each $i\in[m]$, $G_i=G_{i-1}+v_iw_i$ for some non-adjacent vertices $v_i$ and $w_i$ of~$G_{i-1}$.
In other words, from an integrally completable graph, we can make a complete graph by recursively adding new edges with spectral integral variation in a specific way.
By using the characterization in~\cite{kirkland04}, Kirkland~\cite{kirkland05} showed that a graph is integrally completable if and only if it has no induced subgraph isomorphic to $P_4$ (an induced path on four vertices) or $2K_2$ (the disjoint union of two copies of $K_2$).

We generalize the integrally completability to signed graphs.
For a signed graph $(K_n,\Sigma)$, not necessarily Laplacian integral, an $n$-vertex signed graph $(G,\Sigma')$ is \emph{integrally $\Sigma$-completable} if there exists a sequence $(G_0,\Sigma_0),\ldots,(G_m,\Sigma_m)$ of signed graphs such that $(G_0,\Sigma_0)=(G,\Sigma')$, $(G_m,\Sigma_m)=(K_n,\Sigma)$, and for each $i\in[m]$, there exists $v_iw_i\in E(K_n)-E(G_{i-1})$ such that $(G_i,\Sigma_i)$ is obtained from $(G_{i-1},\Sigma_{i-1})$ by adding a new (odd or even) edge with spectral integral variation.
Since~$K_n$ is Laplacian integral, observe that a graph~$G$ is integrally completable if and only if a signed graph $(G,\emptyset)$ is integrally $\emptyset$-completable.
In this manner, integrally $\Sigma$-completable signed graphs generalizes integrally completable graphs.

We fully characterize integrally $\Sigma$-completable graphs for every signed graph $(K_n,\Sigma)$.
For a signed graph $(G,\Gamma)$, we say that a triangle~$uvw$ of~$G$ is \emph{odd} (resp. \emph{even}) if $\abs{\Gamma\cap\{uv,uw,vw\}}$ is odd (resp. even).
We denote a triangle of a graph on vertices~$u$, $v$, and~$w$ by~$uvw$.

\begin{theorem}\label{thm:main3}
    Let $(K_n,\Sigma)$ be a signed graph with $n\geq4$, let $X(\Sigma)$ be the set of edges~$vw$ of~$K_n$ such that $uvw$ is even for every $u\in V(K_n)-\{v,w\}$, and let $Y(\Sigma)$ be the set of edges~$vw$ of~$K_n$ such that for every $u\in V(K_n)-\{v,w\}$, $uvw$ is odd, and the number of odd triangles containing~$u$ and~$v$ is one more than that of even triangles containing~$u$ and~$v$.
    Then an $n$-vertex signed graph $(G,\Sigma')$ is integrally $\Sigma$-completable if and only if $E(K_n)-E(G)\subseteq X(\Sigma)\cup Y(\Sigma)$, $\Sigma'=\Sigma\cap E(G)$, and $K_n\setminus(X(\Sigma)-E(G))$ is integrally completable.
\end{theorem}

We remark that every signed graph on at most~$3$ vertices is Laplacian integral, so adding any new edge yields spectral integral variation.
Thus, for a signed graph $(K_n,\Sigma)$ with $n\leq3$, the integrally $\Sigma$-completable graphs are nothing but the graphs $(K_n\setminus X,\Sigma-X)$ for $X\subseteq E(K_n)$.

Here is the outline of this paper.
In~\Cref{sec:siv}, we prove Theorems~\ref{thm:main1 even} and~\ref{thm:main2 even} as well as their analogues for the cases of adding odd edges.
In~\Cref{sec:completable}, we prove Theorem~\ref{thm:main3}.

\section{Spectral integral variation for signed graphs}\label{sec:siv}

We prove Theorems~\ref{thm:main1 even} and~\ref{thm:main2 even} in Subsections~\ref{subsec:siv one} and~\ref{subsec:siv two}, respectively.
The cases of adding odd edges will follow as their corollaries.
For $k\geq1$, we denote by $e_k$ the $k$-th standard unit basis vector.

\subsection{Spectral integral variation of type~1}\label{subsec:siv one}

We recall that if spectral integral variation of type~1 occurs under an edge addition, then the signed Laplacian matrix of the given signed graph has a unique eigenvalue which will be increased by $2$ after the edge addition.
To prove~\Cref{thm:main1 even}, we use the following theorem.

\begin{theorem}[So~{\cite[Theorem~1.4]{so99}}]\label{thm:symmetric}
    Let $A$ and $B$ be symmetric matrices with spectra, not necessary in descending order, $\{\alpha_1,\ldots,\alpha_n\}$ and $\{\beta,0,\ldots,0\}$, respectively.
    Then $A+B$ has eigenvalues $\{\alpha_1+\beta,\alpha_2,\ldots,\alpha_n\}$ if and only if $AB=BA$.
\end{theorem}

\begin{proof}[Proof of \Cref{thm:main1 even}]
    We apply~\Cref{thm:symmetric} for $A:=L(G,\Sigma)$ where the first and the second columns are indexed by $v$ and $w$, respectively, and $B:=(e_1-e_2)(e_1-e_2)^t$.
    Note that both $A$ and $B$ are symmetric, $L(G+vw,\Sigma)=A+B$, and the spectrum of~$B$ is~$\{2,0,\ldots,0\}$.
    By~\Cref{thm:symmetric}, the spectral integral variation of type~1 occurs if and only if $AB=BA$.
    It implies that the first two columns of $L(G,\Sigma)$ are identical except for the first two entries, and therefore $N^-(v)=N^-(w)$ and $N^+(v)=N^+(w)$.
\end{proof}

As a corollary of \Cref{thm:main1 even}, we characterize the case when spectral integral variation of type~1 occurs under the addition of an odd edge.

\begin{corollary}
    For a signed graph $(G,\Sigma)$ and non-adjacent vertices~$v$ and~$w$ of~$G$, spectral integral variation of type~1 occurs under the addition of an odd edge $vw$ if and only if $N^-(v)=N^+(w)$ and $N^+(v)=N^-(w)$.
\end{corollary}
\begin{proof}
    Let $(G,\Sigma')$ be the signed graph obtained from~$(G,\Sigma)$ by switching at~$w$.
    Since $(G+vw,\Sigma')$ is equivalent to $(G+vw,\Sigma\cup\{vw\})$, spectral integral variation of type~1 occurs in~$(G,\Sigma)$ under the addition of an odd edge~$vw$ if and only if that occurs in $(G,\Sigma')$ under the addition of an even edge $vw$.
    Since $N_{G,\Sigma'}^\pm(w)=N_{G,\Sigma}^\mp(w)$, by applying~\Cref{thm:main1 even} to $(G,\Sigma')$, the statement follows.
\end{proof}

\subsection{Spectral integral variation of type~2}\label{subsec:siv two}

Throughout this subsection, for a signed graph $(G,\Sigma)$ and non-adjacent vertices~$v$ and~$w$ of~$G$, we define
\begin{align*}
    a_{vw}&:=\abs{N(v)-N(w)},\\
    b_{vw}&:=\abs{N(w)-N(v)},\\
    c_{vw}&:=\abs{N^+(v)\cap N^+(w)}+\abs{N^-(v)\cap N^-(w)},\\
    d_{vw}&:=\abs{N^+(v)\cap N^-(w)}+\abs{N^-(v)\cap N^+(w)},\\
    t_{vw}&:=c_{vw}-d_{vw}.
\end{align*}
Observe that the degree of~$v$ in~$G$ is $a_{vw}+c_{vw}+d_{vw}$ and that of~$w$ is $b_{vw}+c_{vw}+d_{vw}$.

We begin with the following lemma to examine the relationships between the degrees of $v$ and $w$ and the eigenvalues of $L(G,\Sigma)$ affected by adding $vw$.
This is an analogue of~{\cite[Lemma~9]{yizheng02}} on signed graphs.
The proof basically follows that of~{\cite[Lemma~9]{yizheng02}}.

\begin{lemma}\label{le_siv2_3}
    Let $(G,\Sigma)$ be a signed graph and $v,w$ be non-adjacent vertices of~$G$ with degrees $d_1$ and $d_2$, respectively.
    Suppose that spectral integral variation of type~2 occurs under the addition of an even edge $vw$, and $\lambda_1,\lambda_2$ are the eigenvalues of $L(G,\Sigma)$ that are increased by~$1$.
    Then the following hold:
    \begin{enumerate}[label=(\roman*)]
        \item\label{siv2_3_1} $\lambda_1+\lambda_2=d_1+d_2+1$,
        \item\label{siv2_3_2} $\lambda_1\lambda_2=d_1d_2+t_{vw}$, and
        \item\label{siv2_3_3} $d_1+d_2-2t_{vw}=a_{vw}+b_{vw}+4d_{vw}>0$.
    \end{enumerate}
\end{lemma}
\begin{proof}
    We may assume that the first two columns of $L:=L(G,\Sigma)$ are indexed by~$v$ and~$w$.
    Let $A:=(e_1-e_2)(e_1-e_2)^t$.
    Note that $L(G+vw,\Sigma)=L+A$.

    We first show~\ref{siv2_3_1}.
    On the one hand, $tr((L+A)^2)-tr(L^2)=(\lambda_1+1)^2+(\lambda_2+1)^2-\lambda_1^2-\lambda_2^2=2(\lambda_1+\lambda_2)+2$.
    On the other hand $tr((L+A)^2)-tr(L^2)=2tr(LA)+tr(A^2)=2(d_1+d_2)+4$.
    Thus,~\ref{siv2_3_1} holds.

    We now show~\ref{siv2_3_3}.
    Since $d_1=a_{vw}+c_{vw}+d_{vw}$, $d_2=b_{vw}+c_{vw}+d_{vw}$, and $t_{vw}=c_{vw}-d_{vw}$, we have that $d_1+d_2-2t_{vw}=a_{vw}+b_{vw}+4d_{vw}\geq0$.
    Suppose for contradiction that $d_1+d_2-2t_{vw}=0$.
    Then $a_{vw}=b_{vw}=d_{vw}=0$, so by~\Cref{thm:main1 even}, spectral integral variation of type~1 occurs under the addition of an even edge $vw$, a contradiction.

    Finally, we show~\ref{siv2_3_2}.
    On the one hand, $tr((L+A)^3)-tr(L^3)=(\lambda_1+1)^3+(\lambda_2+1)^3-\lambda_1^3-\lambda_2^3=3(\lambda_1^2+\lambda_2^2)+3(\lambda_1+\lambda_2)+2=3(\lambda_1^2+\lambda_2^2)+3(d_1+d_2)+5$ by~\ref{siv2_3_1}.
    On the other hand,
    \begin{align*}
        tr((L+A)^3)-tr(L^3)
        =&3tr(LA^2)+3tr(L^2A)+tr(A^3)\\
        =&6(d_1+d_2)+3(d_1^2+d_2^2+a_{vw}+b_{vw}+4d_{vw})+8\\
        =&3(d_1^2+d_2^2)+9(d_1+d_2)-6t_{vw}+8
    \end{align*}
    where the last equality holds by~\ref{siv2_3_3}.
    Thus,
    \[
        \lambda_1^2+\lambda_2^2=(d_1^2+d_2^2)+2(d_1+d_2)-2t_{vw}+1,
    \]
    and therefore by~\ref{siv2_3_1},
    \begin{align*}
        2\lambda_1\lambda_2
        &=(\lambda_1+\lambda_2)^2-(\lambda_1^2+\lambda_2^2)\\
        &=(d_1+d_2+1)^2-(d_1^2+d_2^2)-2(d_1+d_2)+2t_{vw}-1=2(d_1d_2+t_{vw}).\qedhere
    \end{align*}
\end{proof}

Kirkland~{\cite[Theorem~2.2]{kirkland04}} showed that for a graph $G$ and non-adjacent vertices~$v$ and~$w$ of~$G$, spectral integral variation of type~2 occurs under the addition of an edge $vw$ if and only if there exists a specific orthonormal basis of eigenvectors of $L(G)$.
We point out that the same proof of the theorem also works in the signed graph setting, so we restate the theorem in the signed graph setting as follows.
For computational convenience, we slightly modify the statement by considering orthogonal basis, instead of orthonormal one.

\begin{proposition}\label{prop:siv2 eigenvectors}
    Let $(G,\Sigma)$ be an $n$-vertex signed graph and $v,w$ be non-adjacent vertices of~$G$.
    Suppose that the first two columns of $L(G,\Sigma)$ are indexed by~$v$ and~$w$.
    Then spectral integral variation of type~2 occurs under the addition of an even edge $vw$ by increasing eigenvalues $\lambda_1,\lambda_2$ of $L(G,\Sigma)$ by~$1$ if and only if there exists an orthogonal basis of eigenvectors $v_1,\ldots,v_n$ of $L(G,\Sigma)$ satisfying the following conditions:
    \begin{enumerate}[label=(\roman*)]
        \item\label{siv2_1_1} for each $i\in[2]$, $v_i$ is a $\lambda_i$-eigenvector of the form
            \[
                v_i=\left[\begin{array}{c}
                    a_i\\
                    b_i\\\hline
                    u_i
                \end{array}\right]
            \]
            where $(a_i-b_i)^2/\|v_i\|^2=1+1/(\lambda_{3-i}-\lambda_i)$, and
        \item\label{siv2_1_2} for each $j\in[n]-[2]$, the first two entries of $v_j$ are equal.
    \end{enumerate}
\end{proposition}

As a corollary of~\Cref{prop:siv2 eigenvectors} in the graph setting, Kirkland~{\cite[Corollary 2.5]{kirkland04}} described the subcolumn~$u_i$ of~$v_i$ for $i\in[2]$ in terms of subcolumns of the Laplacian matrix of a graph.
In the following corollary, we present an analogous result in the signed graph setting.
However, we remark that the proof of~{\cite[Corollary~2.5]{kirkland04}} is no longer applicable to the signed graph setting.
The reason is that Kirkland~\cite{kirkland04} crucially used an observation that~$\mathbf{1}$ is a null vector of the Laplacian matrix of a graph, which is generally false in signed graphs.
We present an independent proof for the analogous corollary.
This proof indeed yields simplified proofs of~{\cite[Corollary~2.5]{kirkland04}} and~{\cite[Theorem~2.6]{kirkland04}}.

\begin{corollary}\label{co:siv2 eigenvectors}
    Let $(G,\Sigma)$ be an $n$-vertex signed graph and $v,w$ be non-adjacent vertices of~$G$.
    Suppose that the first two columns of $L(G,\Sigma)$ are indexed by~$v$ and~$w$ and that there exists an orthogonal basis of eigenvectors $v_1,\ldots,v_n$ of~$L(G,\Sigma)$ satisfying~\Cref{prop:siv2 eigenvectors}\ref{siv2_1_1}--\ref{siv2_1_2}.
    Let $\lambda_1,\lambda_2$ be the eigenvalues of $L(G,\Sigma)$ increased by~$1$ after adding an even edge $vw$ to~$(G,\Sigma)$.
    For each $i\in[2]$, let $x_i$ be the vector obtained from the $i$-th column of $L(G,\Sigma)$ by removing the first two entries.
    For each $j\in[n]$, let $u_j$ be the vector obtained from $v_j$ by removing the first two entries.
    Then there exist non-zero constants $c_1$ and $c_2$ such that for each $i\in[2]$, $u_i=c_i(x_2-x_1)$.
    In particular, one can choose the orthogonal basis so that both~$c_1$ and~$c_2$ are~$1$.
\end{corollary}
\begin{proof}
    By~\Cref{prop:siv2 eigenvectors}\ref{siv2_1_2}, for each $j\in[n]-[2]$, the first two entries of $v_j$ are same.
    Thus, by scaling if necessary, we may assume that for each $j\in[n]-[2]$, the first entry of~$v_j$ is either~$1$ or~$0$.
    Furthermore, by relabeling if necessary, we may assume that there exists $m\in[n]-[1]$ such that for each $j\in[n]-[2]$, the first entry of~$v_j$ is~$1$ if and only if $j\leq m$.

    \begin{claim}\label{cl_siv2_2_1}
        Neither $u_1$ nor $u_2$ is $\mathbf{0}$.
    \end{claim}
    \begin{subproof}
        Suppose not.
        Since $v_1$ and $v_2$ span $e_1-e_2$, there exist scalars $t_1$ and $t_2$ with $t_1v_1+t_2v_2=e_1-e_2$.

        We first assume that $u_1=u_2=\mathbf{0}$.
        For each $i\in[2]$, since $v_i$ is a $\lambda_i$-eigenvector of $L(G,\Sigma)$, we have that $a_ix_1+b_ix_2=\mathbf{0}$.
        By multiplying $L(G,\Sigma)$ on both sides of $t_1v_1+t_2v_2=e_1-e_2$, we deduce that
        \[
            \mathbf{0}=t_1(a_1x_1+b_1x_2)+t_2(a_2x_1+b_2x_2)=x_1-x_2.
        \]
        It means that $N^-(v)=N^-(w)$ and $N^+(v)=N^+(w)$.
        Thus, by~\Cref{thm:main1 even}, the spectral integral variation is of type~1, a contradiction.

        Hence, exactly one of~$u_1$ and~$u_2$, say $u_1$, is~$\mathbf{0}$.
        Since $u_2\neq\mathbf{0}$, we deduce from $t_1v_1+t_2v_2=e_1-e_2$ that $t_2=0$, and therefore $a_1=-b_1$.
        Since $v_1$ is a $\lambda_1$-eigenvector of $L(G,\Sigma)$, we have that $a_1x_1+b_1x_2=\mathbf{0}$.
        Hence, $x_1-x_2=\mathbf{0}$, a contradiction.
    \end{subproof}

    If $m=2$, then $\{u_j:j\in[n]-[2]\}$ is linearly independent, so by the orthogonality, $u_1=u_2=\mathbf{0}$, contradicting~\Cref{cl_siv2_2_1}.
    Thus, $m\geq3$.
    Let $U:= \{u_j-u_3:j\in[m]-[3]\}\cup\{u_k:k\in[n]-[m]\}$.

    We show that $U$ is linearly independent.
    Since $\{v_j:j\in[m]-[3]\}$ is linearly independent, we have that $\{u_j:j\in[m]-[2]\}$ is affinely independent, which is equivalent to that $\{u_j-u_3:j\in[m]-[3]\}$ is linearly independent.
    We may assume that $m<n$, because otherwise we are done.
    For each $k\in[n]-[m]$ and every $\ell\in[n]$ distinct to $k$, since the first two entries of $v_k$ are $0$, we have that $u_k^tu_\ell=v_k^tv_\ell=0$.
    Hence, we deduce that $U$ is linearly independent.

    Since the rank of $U$ is $n-3$ and every vector in $U$ consists of $n-2$ entries, the orthogonal complement of the space spanned by $U$ has rank~$1$.
    Therefore, to show that for each $i\in[2]$, there exists $c_i$ with $u_i=c_i(x_2-x_1)$, it suffices to show that all of $u_1$, $u_2$, and $x_2-x_1$ are orthogonal to the vectors in~$U$.
    Note that such $c_i$ is non-zero by~\Cref{cl_siv2_2_1}.

    Let $i\in[2]$.
    We first show that $u_i$ is orthogonal to the vectors in~$U$.
    For each $j\in[m]-[3]$, since the first two entries of~$v_j-v_3$ are~$0$, we have that $u_i^t(u_j-u_3)=v_i^t(v_j-v_3)=0$.
    For each $k\in[n]-[m]$, since the first two entries of~$v_k$ are~$0$, we have that $u_i^tu_k=v_i^tv_k=0$.
    Therefore, $u_i$ is orthogonal to the vectors in~$U$.

    We now show that $x_2-x_1$ is orthogonal to the vectors in~$U$.
    For each $j\in[m]-[2]$, since the first two entries of~$v_j$ are~$1$ and $v_j$ is an eigenvector of $L(G,\Sigma)$, we have that $d_1-x_1^tu_j=d_2-x_2^tu_j$.
    Then $(x_2-x_1)^tu_j=d_2-d_1$, so that $(x_2-x_1)^t(u_j-u_3)=0$.
    For each $k\in[n]-[m]$, since the first two entries of~$v_k$ are~$0$ and $v_k$ is an eigenvector of $L(G,\Sigma)$, we have that $x_1^tu_k=x_2^tu_k=0$.
    Thus, $x_2-x_1$ is also orthogonal to the vectors in~$U$, and therefore for each $i\in[2]$, there exists $c_i\neq0$ with $u_i=c_i(x_2-x_1)$.
    By scaling if necessary, we can take the eigenvectors~$v_1$ and~$v_2$ of~$L(G,\Sigma)$ so that $c_1=c_2=1$.
\end{proof}

For a signed graph with a non-complete underlying graph, we now discuss how to obtain an equivalent signed graph whose signed Laplacian matrix is given as in~\eqref{eq:centered}.
Let $(G,\Sigma)$ be a signed graph and $v,w$ be non-adjacent vertices of $G$.
We say that $(G,\Sigma)$ is \emph{$(v,w)$-centered} if $N^-(v)=\emptyset$ and $N^-(w)\subseteq N(v)\cap N(w)$.
Observe that if $(G,\Sigma)$ is $(v,w)$-centered, then $L(G,\Sigma)$ can be arranged as~\eqref{eq:centered}.
In the following observation, we present a way of converting a signed graph with a non-complete underlying graph into an equivalent $(v,w)$-centered signed graph.
This will help us simplify the characterization of spectral integral variation.

\begin{observation}\label{obs:centered}
    Let $(G,\Sigma)$ be a signed graph and $v,w$ be non-adjacent vertices of~$G$.
    Let $(G,\Sigma')$ be a signed graph obtained from $(G,\Sigma)$ by switching at each vertex in $N_{G,\Sigma}^-(v)\cup(N_{G,\Sigma}^-(w)-N_G(v))$, in any order.
    Then $(G,\Sigma')$ is $(v,w)$-centered.
\end{observation}

In the following lemma, we explicitly describes the eigenvectors $v_1$ and $v_2$ in~\Cref{co:siv2 eigenvectors} when $(G,\Sigma)$ is $(v,w)$-centered.

\begin{lemma}\label{le_siv2_final}
    Let $(G,\Sigma)$ be a $(v,w)$-centered signed graph for non-adjacent vertices~$v$ and~$w$ of~$G$ with degrees~$d_1$ and~$d_2$, respectively.
    Suppose that $L(G,\Sigma)$ is given as~\eqref{eq:centered} and that there is an orthogonal basis of eigenvectors $v_1,\ldots,v_n$ of $L(G,\Sigma)$ satisfying~\Cref{prop:siv2 eigenvectors}\ref{siv2_1_1}--\ref{siv2_1_2} such that the constants $c_1$ and $c_2$ in~\Cref{co:siv2 eigenvectors} are~$1$.
    Let $\lambda_1,\lambda_2$ be the eigenvalues of $L(G,\Sigma)$ increased by~$1$ after adding an even edge~$vw$ to $(G,\Sigma)$.
    Then for each $i\in[2]$,
    \begin{enumerate}[label=(\roman*)]
        \item\label{final-1} $a_i=-\lambda_i+d_2+1$ and
        \item\label{final-2} $b_i=\lambda_i-d_1-1$.
    \end{enumerate}
    That is,
    \begin{equation}
        v_1=\left[\begin{array}{c}
            -\lambda_1+d_2+1\\
            \lambda_1-d_1-1\\\hline
            \mathbf{1}\\\hline
            -\mathbf{1}\\\hline
            \mathbf{0}\\\hline
            2\cdot\mathbf{1}\\\hline
            \mathbf{0}
        \end{array}\right]
        \text{\qquad and\qquad}
        v_2=\left[\begin{array}{c}
            -\lambda_2+d_2+1\\
            \lambda_2-d_1-1\\\hline
            \mathbf{1}\\\hline
            -\mathbf{1}\\\hline
            \mathbf{0}\\\hline
            2\cdot\mathbf{1}\\\hline
            \mathbf{0}
        \end{array}\right].
        \label{eigenvectors}
    \end{equation}
\end{lemma}
\begin{proof}
    We first show that for each $i\in[2]$,
    \begin{equation}
        (a_i-b_i)^2=(a_i^2+b_i^2+d_1+d_2-2t_{vw})\left(1+\frac{1}{\lambda_{3-i}-\lambda_i}\right).\label{eq:step1}
    \end{equation}
    By the assumption and~\Cref{co:siv2 eigenvectors}, $v_i^t=(a_i,b_i,\mathbf{1}^t,-\mathbf{1}^t,\mathbf{0},2\cdot\mathbf{1}^t,\mathbf{0}^t)$.
    By~\Cref{le_siv2_3}\ref{siv2_3_3}, we have that
    \[
        \|v_i\|^2=a_i^2+b_i^2+a_{vw}+b_{vw}+4d_{vw}=a_i^2+b_i^2+d_1+d_2-2t_{vw}.
    \]
    Then~\eqref{eq:step1} follows from~\Cref{prop:siv2 eigenvectors}\ref{siv2_1_1}.

    Next, we show that for each $i\in[2]$,
    \begin{align}
        (\lambda_i-d_1)a_i&=-d_1+t_{vw},\label{eq:step2-1}\\
        (\lambda_i-d_2)b_i&=d_2-t_{vw}.\label{eq:step2-2}
    \end{align}
    Since $v_i$ is a $\lambda_i$-eigenvector of $L(G,\Sigma)$, we have that $a_id_1-a_{vw}-2d_{vw}=\lambda_ia_i$.
    Since $d_1=a_{vw}+c_{vw}+d_{vw}$, we deduce that $(\lambda_i-d_1)a_i=-a_{vw}-2d_{vw}=-d_1+t_{vw}$.
    Similarly, as $v_i$ is a $\lambda_i$-eigenvector of $L(G,\Sigma)$, we have that $b_id_2+b_{vw}+2d_{vw}=\lambda_ib_i$.
    Since $d_2=b_{vw}+c_{vw}+d_{vw}$, we deduce that $(\lambda_i-d_2)b_i=d_2-t_{vw}$.

    We now show \ref{final-1}--\ref{final-2} for a fixed $i\in[2]$.
    We first show~\ref{final-1}.
    We have that
    \begin{align*}
        (\lambda_i-d_1)(-\lambda_i+d_2+1)
        &=\lambda_i(-\lambda_i+d_1+d_2+1)-d_1d_2-d_1\\
        &=\lambda_1\lambda_2-d_1d_2-d_1 \tag*{by~\Cref{le_siv2_3}\ref{siv2_3_1}}\\
        &=-d_1+t_{vw} \tag*{by~\Cref{le_siv2_3}\ref{siv2_3_2}}\\
        &=(\lambda_i-d_1)a_i \tag*{by~\eqref{eq:step2-1}.}
    \end{align*}
    If $\lambda_i\neq d_1$, then~\ref{final-1} directly follows.
    Thus, we may assume that $\lambda_i=d_1$.
    Then~\Cref{le_siv2_3}\ref{siv2_3_1} and~\ref{siv2_3_2} imply that $\lambda_{3-i}=d_2+1$ and $t_{vw}=d_1$, respectively.
    Note that $d_2-d_1=d_1+d_2-2t_{vw}>0$ by~\Cref{le_siv2_3}\ref{siv2_3_3}.
    Hence, \eqref{eq:step2-1} implies that $a_{3-i}=0$, and~\eqref{eq:step2-2} implies that $b_i=-1$ and $b_{3-i}=d_2-d_1$.
    Then by~\eqref{eq:step1},
    \begin{align*}
        (a_i+1)^2
        &=(a_i^2+1+d_2-d_1)\left(1+\frac{1}{d_2-d_1+1}\right)\\
        &=(a_i^2+1+b_{3-i})\cdot\frac{b_{3-i}+2}{b_{3-i}+1}.
    \end{align*}
    Therefore,
    \[
        0=(a_i^2+1+b_{3-i})(b_{3-i}+2)-(a_i+1)^2(b_{3-i}+1)=(b_{3-i}+1-a_i)^2,
    \]
    which means that $a_i=b_{3-i}+1=-\lambda_1+d_2+1$.

    Similarly, we show~\ref{final-2}.
    First, we have that
    \begin{align*}
        (\lambda_i-d_2)(\lambda_i-d_1-1)
        &=\lambda_i(\lambda_i-d_1-d_2-1)+d_1d_2+d_2\\
        &=-\lambda_1\lambda_2+d_1d_2+d_2 \tag*{by~\Cref{le_siv2_3}\ref{siv2_3_1}}\\
        &=d_2-t_{vw} \tag*{by~\Cref{le_siv2_3}\ref{siv2_3_2}}\\
        &=(\lambda_i-d_2)b_i \tag*{by~\eqref{eq:step2-2}.}
    \end{align*}
    If $\lambda_i\neq d_2$, then~\ref{final-2} directly follows.
    Thus, we may assume that $\lambda_i=d_2$.
    Then~\Cref{le_siv2_3}\ref{siv2_3_1} and~\ref{siv2_3_2} imply that $\lambda_{3-i}=d_1+1$ and $t_{vw}=d_2$, respectively.
    Note that $d_1-d_2=d_1+d_2-2t_{vw}>0$ by~\Cref{le_siv2_3}\ref{siv2_3_3}.
    Hence, \eqref{eq:step2-2} implies that $b_{3-i}=0$, and~\eqref{eq:step2-1} implies that $a_i=1$ and $a_{3-i}=d_2-d_1$.
    Then by~\eqref{eq:step1},
    \begin{align*}
        (1-b_i)^2
        &=(b_i^2+1+d_1-d_2)\left(1+\frac{1}{d_1-d_2+1}\right)\\
        &=(b_i^2+1-a_{3-i})\cdot\frac{2-a_{3-i}}{1-a_{3-i}}.
    \end{align*}
    Therefore,
    \[
        0=(b_i^2+1-a_{3-i})(2-a_{3-i})-(1-b_i)^2(1-a_{3-i})=(1-a_{3-i}+b_i)^2,
    \]
    which means that $b_i=a_{3-i}-1=\lambda_i-d_1-1$.
\end{proof}

We now prove~\Cref{thm:main2 even}.

\begin{proof}[Proof of~\Cref{thm:main2 even}]
    Suppose first that~\eqref{eq:main 1}--\eqref{eq:main 5} hold.
    Consider $x^2-(d_1+d_2+1)x+(d_1d_2+t_{vw})=0$.
    Since
    \begin{align*}
        &(d_1+d_2+1)^2-4(d_1d_2+t_{vw})\\
        &=(a_{vw}+b_{vw}+2c_{vw}+2d_{vw}+1)^2-4(a_{vw}+c_{vw}+d_{vw})(b_{vw}+c_{vw}+d_{vw})-4(c_{vw}-d_{vw})\\
        &=(a_{vw}-b_{vw})^2+2a_{vw}+2b_{vw}+8d_{vw}+1>0,
    \end{align*}
    the equation has two real roots $\eta_1$ and $\eta_2$.
    Note that $\eta_1+\eta_2=d_1+d_2+1$ and $\eta_1\eta_2=d_1d_2+t_{vw}$.
    Let
    \[
        w_1=\left[\begin{array}{c}
            -\eta_1+d_2+1\\
            \eta_1-d_1-1 \\\hline
            \mathbf{1}\\\hline
            -\mathbf{1}\\\hline
            \mathbf{0}\\\hline
            2\cdot\mathbf{1}\\\hline
            \mathbf{0}
        \end{array}\right]
        \text{\qquad and\qquad}
        w_2=\left[\begin{array}{c}
            -\eta_2+d_2+1\\
            \eta_2-d_1-1 \\\hline
            \mathbf{1}\\\hline
            -\mathbf{1}\\\hline
            \mathbf{0}\\\hline
            2\cdot\mathbf{1}\\\hline
            \mathbf{0}
        \end{array}\right].
    \]
    Since~\eqref{eq:main 1}--\eqref{eq:main 5} hold, by direct computations, one can see that for each $i\in[2]$, $w_i$ is an $\eta_i$-eigenvector of $L(G,\Sigma)$, and that $w_i-(e_1-e_2)$ is an $(\eta_i+1)$-eigenvector of $L(G+vw,\Sigma)=L(G,\Sigma)+(e_1-e_2)(e_1-e_2)^t$.
    Therefore, spectral integral variation of type~2 occurs under the addition of an even edge $vw$.
    
    Conversely, suppose that spectral integral variation of type~2 occurs under the addition of an even edge~$vw$ by increasing $\lambda_1,\lambda_2$ of~$L$.
    By~\Cref{prop:siv2 eigenvectors} and~\Cref{le_siv2_final}, there is an orthogonal basis of eigenvectors $v_1,\ldots,v_n$ of $L(G,\Sigma)$ such that for each $i\in[2]$, $v_i$ is a $\lambda_i$-eigenvector of the form in~\eqref{eigenvectors}.
    Since $v_i$ is a $\lambda_i$-eigenvector, by direct computations, we deduce that~\eqref{eq:main 1}--\eqref{eq:main 5} hold.
\end{proof}

As a corollary, we characterize when spectral integral variation of type~2 occurs under the addition of an odd edge.

\begin{corollary}
    Let $(G,\Sigma)$ be a signed graph and $v$ and $w$ be non-adjacent vertices of~$G$ with degrees~$d_1$ and~$d_2$, respectively.
    If $L(G,\Sigma)$ is given as~\eqref{eq:centered}, then spectral integral variation of type~2 occurs under the addition of an odd edge $vw$ if and only if the following conditions hold:
    \begin{align*}
        L_{11}\mathbf{1}+L_{12}\mathbf{1}+2L_{13}\mathbf{1}&=(d_2+1)\mathbf{1},\\
        L_{21}\mathbf{1}+L_{22}\mathbf{1}+2L_{23}\mathbf{1}&=(d_1+1)\mathbf{1},\\
        L_{31}\mathbf{1}+L_{32}\mathbf{1}+2L_{33}\mathbf{1}&=(d_1+d_2+2)\mathbf{1},\\
        L_{41}\mathbf{1}+L_{42}\mathbf{1}+2L_{43}\mathbf{1}&=-(d_1-d_2)\mathbf{1},\\
        L_{51}\mathbf{1}+L_{52}\mathbf{1}+2L_{53}\mathbf{1}&=\mathbf{0}.
    \end{align*}
\end{corollary}
\begin{proof}
    Let $(G,\Sigma')$ be the signed graph obtained from~$(G,\Sigma)$ by switchings at the vertices in $(N(w)-N(v))\cup\{w\}$, in any order.
    Note that $(G+vw,\Sigma')$ is equivalent to $(G+vw,\Sigma\cup\{vw\})$ via the same switchings.
    Thus, spectral integral variation of type~2 occurs in~$(G,\Sigma)$ under the addition of an odd edge~$vw$ if and only if that occurs in $(G,\Sigma')$ under the addition of an even edge $vw$.
    Note that $L(G,\Sigma')$ can be arranged as
    \[
        \left[\begin{array}{rr|r|r|r|r|r}
            d_1 & 0 & -\mathbf{1}^t & \mathbf{0}^t & -\mathbf{1}^t & -\mathbf{1}^t & \mathbf{0}^t\\
            0 & d_2 & \mathbf{0}^t & -\mathbf{1}^t & -\mathbf{1}^t & \mathbf{1}^t & \mathbf{0}^t\\\hline
            -\mathbf{1} & \mathbf{0} & L_{11} & -L_{12} & L_{14} & L_{13} & L_{15}\\\hline
            \mathbf{0} & -\mathbf{1} & -L_{21} & L_{22} & -L_{24} & -L_{23} & -L_{25}\\\hline
            -\mathbf{1} & -\mathbf{1} & L_{41} & -L_{42} & L_{44} & L_{43} & L_{45}\\\hline 
            -\mathbf{1} & \mathbf{1} & L_{31} & -L_{32} & L_{34} & L_{33} & L_{35}\\\hline
            \mathbf{0} & \mathbf{0} & L_{51} & -L_{52} & L_{54} & L_{53} & L_{55}
        \end{array}\right].
    \]
    Then the statement follows from~\Cref{thm:main2 even}.
\end{proof}

\section{Integrally $\Sigma$-completable signed graphs}\label{sec:completable}

In this section, we prove~\Cref{thm:main3}.
Here is our strategy.
In~\Cref{subsec:structure}, we show that for a signed graph $(K_n,\Sigma)$ with $n\geq4$, the graph on $V(K_n)$ with edge set $X(\Sigma)\cup Y(\Sigma)$ consists of complete components.
Based on this structure, in~\Cref{subsec:substitution}, we introduce a vertex substitution operation and analyze the spectrum of the signed graph obtained by this operation.
With this spectrum analysis, in~\Cref{subsec:siv}, we characterize circumstances when we can add a new edge not in~$X(\Sigma)$ to an integrally $\Sigma$-completable graph with spectral integral variation.
Finally, in~\Cref{subsec:main proof}, we complete the proof of~\Cref{thm:main3} using this characterization.

\subsection{Structures of $X(\Sigma)$ and $Y(\Sigma)$}\label{subsec:structure}

We will use the following simple observation often in several proofs.
The first one is set-theoretically obvious, and the second one directly follows from the parity conditions of $X(\Sigma)$ and $Y(\Sigma)$.

\begin{observation}\label{obs:partition}
    For a signed graph $(G,\Sigma)$ and subsets~$A$ and~$B$ of~$E(G)$, $\abs{\Sigma\cap A}$ is even if and only if $\abs{\Sigma\cap B}\equiv\abs{\Sigma\cap(A\triangle B)}\pmod{2}$.    
    In particular, if $G=K_4$ with vertex set $\{v_1,\ldots,v_4\}$, then $v_1v_2v_3$ and $v_1v_2v_4$ have the same parity if and only if $v_1v_3v_4$ and $v_2v_3v_4$ have the same parity.
\end{observation}

\begin{observation}\label{obs:no common}
    For a signed graph $(K_n,\Sigma)$ with $n\geq4$, $K_n$ has no vertex incident with an edge in~$X(\Sigma)$ and an edge in~$Y(\Sigma)$.
\end{observation}

We show that $X(\Sigma)$ induces a subgraph of~$K_n$ where every component is complete.

\begin{lemma}\label{lem:friendly} 
    For a signed graph $(K_n,\Sigma)$ with $n\geq4$, the graph on $V(K_n)$ with edge set $X(\Sigma)$ consists of complete components.
\end{lemma}
\begin{proof}
    It suffices to show that for each distinct~$uv$ and~$vw$ in~$X(\Sigma)$, $uw$ is also contained in $X(\Sigma)$.
    Since $uv\in X(\Sigma)$, $uvw$ is even.
    For every $x\in V(K_n)-\{u,v,w\}$, $uwx$ is even as each of $uvw$, $uvx$, and $vwx$ is even.
    Hence, $uw\in X(\Sigma)$.
\end{proof}

We show that $Y(\Sigma)$ contains at most one edge.

\begin{lemma}\label{lem:suspicious}
    For a signed graph $(K_n,\Sigma)$ with $n\geq4$, $\abs{Y(\Sigma)}\leq1$.
\end{lemma}
\begin{proof}
    Suppose for contradiction that $Y(\Sigma)$ contains distinct edges~$v_1w_1$ and~$v_2w_2$.
    Without loss of generality, we may assume that $w_2\notin\{v_1,w_1\}$.
    If~$n$ is even, then the number of odd triangles containing~$w_1$ and~$w_2$ has the same parity with the number of even triangles containing them, contradicting that $v_1w_1\in Y(\Sigma)$.
    Hence,~$n$ is odd.

    We show that~$v_1w_1$ and $v_2w_2$ have no common end.
    Suppose for contradiction that $v_2\in\{v_1,w_1\}$.
    Without loss of generality, we may assume that $v_1=v_2$.
    Since $v_1w_1\in Y(\Sigma)$, $w_2\notin\{v_1,w_1\}$, and $n\geq4$, there are at least $\lfloor(n-2)/2\rfloor\geq1$ even triangles containing~$v_1$ and~$w_2$, contradicting that $v_1w_2=v_2w_2\in Y(\Sigma)$.
    Hence, $v_1w_1$ and $v_2w_2$ have no common end.

    Let $S_1$ be the set of vertices $u\in V(K_n)-\{v_1,v_2,w_1,w_2\}$ such that $uv_1v_2$ is odd, and let $S_2:=V(K_n)-(S_1\cup\{v_1,v_2,w_1,w_2\})$.
    Since $v_1w_1$ and $v_2w_2$ are edges in $Y(\Sigma)$ having no common end, both $v_1v_2w_1$ and $v_1v_2w_2$ are odd triangles, and therefore $\abs{S_2}=\abs{S_1}+1$ by the definition of~$Y(\Sigma)$.

    We first assume that $n=5$.
    Note that~$S_2$ is a singleton and~$S_1$ is empty.
    Let $u$ be the unique vertex in~$S_2$.
    Since $uv_1w_1$ is odd and $uv_1v_2$ is even, by the definition of~$Y(\Sigma)$, $uv_1w_2$ is odd.
    Since $uv_2w_2$ is odd, by~\Cref{obs:partition}, $v_1v_2w_2$ is even, contradicting that $v_2w_2\in Y(\Sigma)$.

    Hence, we may assume that $n\geq7$.
    Note that both~$S_1$ and~$S_2$ are non-empty.
    Since $\abs{S_2}=\abs{S_1}+1$, to derive a contradiction, it suffices to show that both~$\abs{S_1}$ and~$\abs{S_2}$ are even.
    
    We fix $i\in[2]$ and show that~$\abs{S_i}$ is even.
    Since $S_{3-i}$ is non-empty, we can choose a vertex $u\in S_{3-i}$.
    For each $j\in[2]$, let $T_j$ be the set of vertices $u'\in V(K_n)-\{u,v_j\}$ such that $uu'v_j$ is even, and let $T'_j:=V(K_n)-(T_j\cup\{u,v_j\})$.
    By the definition of~$Y(\Sigma)$, $\abs{T_1}=\abs{T_2}=(n-3)/2$.
    Note that $(T_1\cap S_i,T'_1\cap S_i)$ is a partition of $S_i$.
    Thus, to show that $\abs{S_i}$ is even, it suffices to show that $\abs{T_1\cap S_i}\equiv\abs{T'_1\cap S_i}\pmod{2}$.

    For each $j\in[2]$, since both $uv_jw_j$ and $v_1v_2w_j$ are odd, by~\Cref{obs:partition}, $uv_1v_2$ and $uv_{3-j}w_j$ have the same parity.
    Thus, either $\{v_j,w_j\}\subseteq T_{3-j}$ or $\{v_j,w_j\}\subseteq T'_{3-j}$.
    Since $w_{3-j}\in T'_{3-j}$, if $\{v_j,w_j\}\subseteq T_{3-j}$, then $\abs{T_{3-j}}\equiv\abs{T_{3-j}\cap S_1}+\abs{T_{3-j}\cap S_2}\pmod{2}$, and otherwise $\abs{T_{3-j}}\equiv\abs{T_{3-j}\cap S_1}+\abs{T_{3-j}\cap S_2}+1\pmod{2}$.
    Since $\abs{T_1}=\abs{T_2}$, in both cases, we deduce that
    \begin{equation}
        \abs{T_1\cap S_1}+\abs{T_1\cap S_2}\equiv\abs{T_2\cap S_1}+\abs{T_2\cap S_2}\pmod{2}.\label{eq:partition parity}
    \end{equation}

    For every $u'\in S_{3-i}-\{u\}$, since $uv_1v_2$ and $u'v_1v_2$ have the same parity, by~\Cref{obs:partition}, $u'\in T_1$ if and only if $u'\in T_2$.
    It implies that $T_1\cap S_{3-i}=T_2\cap S_{3-i}$, so by~\eqref{eq:partition parity}, $\abs{T_1\cap S_i}\equiv\abs{T_2\cap S_i}\pmod{2}$.
    On the other hand, for every $u''\in S_i$, $uv_1v_2$ and $u''v_1v_2$ have opposite parity.
    Thus, for each $j\in[2]$, by~\Cref{obs:partition}, $u''\in T_j$ if and only if $u''\in T'_{3-j}$, which implies that $T_2\cap S_i=T'_1\cap S_i$.
    Hence, $\abs{T_1\cap S_i}\equiv\abs{T'_1\cap S_i}\pmod{2}$, and therefore $\abs{S_i}$ is even.
    This completes the proof.
\end{proof}

By combining~\Cref{obs:no common} and Lemmas~\ref{lem:friendly}--\ref{lem:suspicious}, we deduce that the graph on $V(K_n)$ with edge set $X(\Sigma)\cup Y(\Sigma)$ consists of complete components.

In the following lemma, we investigate the effects of replacing the sign~$\Sigma$ with~$\Sigma\triangle Y(\Sigma)$.

\begin{lemma}\label{lem:swapping}
    For a signed graph $(K_n,\Sigma)$ with $n\geq4$, if $Y(\Sigma)$ is non-empty, then $X(\Sigma)\cup Y(\Sigma)=X(\Sigma\triangle Y(\Sigma))$ and $Y(\Sigma\triangle Y(\Sigma))=\emptyset$.
\end{lemma}
\begin{proof}
    Let $vw$ be an edge in~$Y(\Sigma)$, which is unique by~\Cref{lem:suspicious}, and let $\Sigma':=\Sigma\triangle\{vw\}$.
    We first show that $X(\Sigma)\cup Y(\Sigma)=X(\Sigma\triangle Y(\Sigma))$.
    For every edge $v'w'\in X(\Sigma)$ and every $x\in V(K_n)-\{v',w'\}$, by~\Cref{obs:no common}, $v'w'x$ has the same parity in~$(K_n,\Sigma)$ and~$(K_n,\Sigma')$, and therefore $v'w'\in X(\Sigma')$.
    On the other hand, since every triangle in $(K_n,\Sigma)$ containing~$v$ and~$w$ is odd, $vw$ is contained in $X(\Sigma')$.
    Therefore, $X(\Sigma)\cup Y(\Sigma)\subseteq X(\Sigma\triangle Y(\Sigma))$.

    Suppose for contradiction that $X(\Sigma\triangle Y(\Sigma))-(X(\Sigma)\cup Y(\Sigma))$ contains an edge $v'w'$.
    If~$vw$ and~$v'w'$ do not share a common end, then $v'w'\in X(\Sigma\triangle\{vw\})$ if and only if $v'w'\in X(\Sigma)$, contradicting that $v'w'\notin X(\Sigma)$.
    Thus, we may assume that $v=v'$.
    Since every triangle containing~$v$ and~$w$ is even in $(K_n,\Sigma\triangle\{vw\})$, $vww'$ is the only odd triangle containing~$v$ and~$w'$ in $(K_n,\Sigma)$, contradicting that $vw\in Y(\Sigma)$.
    Hence, $X(\Sigma)\cup Y(\Sigma)=X(\Sigma\triangle Y(\Sigma))$.

    We now show that $Y(\Sigma')$ is empty.
    Suppose for contradiction that it contains an edge $v'w'$.
    Recall that $vw\in X(\Sigma')$.
    Thus, by~\Cref{obs:no common}, $vw$ and~$v'w'$ have no common end.
    By the definition of~$Y(\Sigma')$, the number of odd triangles in~$(K_n,\Sigma')$ containing~$v$ and~$v'$ is one more than that of even triangles in~$(K_n,\Sigma')$ containing~$v$ and~$v'$.
    Since $vv'w$ is even in~$(K_n,\Sigma')$ and odd in~$(K_n,\Sigma)$, the number of odd triangles in~$(K_n,\Sigma)$ containing~$v$ and~$v'$ is two more than that of even triangles in~$(K_n,\Sigma)$ containing~$v$ and~$v'$, contradicting that $vw\in Y(\Sigma)$.
\end{proof}

\subsection{Vertex substitution}\label{subsec:substitution}

Let $(G,\Sigma)$ and $(H,\Gamma)$ signed graphs with disjoint vertex sets.
For a vertex~$v$ of~$G$, \emph{substituting~$v$ with~$(H,\Gamma)$} is an operation of constructing a new signed graph $(G',\Sigma')$ by taking the disjoint union of~$(G-v,\Sigma-\delta(v))$ and~$(H,\Gamma)$ and adding new edges $\{uv':u\in V(H),v'\in N_G(v)\}$ such that for each $u\in V(H)$ and $v'\in N_G(v)$, the parity of~$uv'$ in $(G',\Sigma')$ is equal to that of~$vv'$ in $(G,\Sigma)$.
In other words,
\[
    \Sigma'=(\Sigma-\{vv':v'\in N_{G,\Sigma}^-(v)\})\cup\Gamma\cup\{uv':u\in V(H),v'\in N_{G,\Sigma}^-(v)\}.
\]

We analyze the spectrum of a signed graph obtained by the substitution operations as follows.

\begin{lemma}\label{lem:substitution}
    Let $(G,\Sigma)$ be the signed graph obtained from a signed graph $(G',\Sigma')$ on $\{v_1,\ldots,v_k\}$ by substituting $v_i$ with a signed graph $(G_i,\Sigma_i)$ for each $i\in[k]$.
    For each $i\in[k]$, let $n_i:=\abs{V(G_i)}$, let $m_i:=\sum_{v_j\in N_{G'}(v_i)}n_j$, and let $\{\lambda_{i,1},\ldots,\lambda_{i,n_i}\}$ be the spectrum of $L(G_i,\Sigma_i)$.
    Let $M:=M(G',\Sigma')$ be a $k\times k$ matrix such that for each~$i,j\in[k]$,
    \[
        M_{i,j}:=\begin{cases}
            \lambda_{i,n_i}+m_i &\text{if $i=j$,}\\
            n_j &\text{if $i\neq j$ and $v_iv_j \in \Sigma'$,}\\
            -n_j &\text{if $i\neq j$ and $v_iv_j \in E(G')-\Sigma'$, and}\\
            0 &\text{otherwise,}
        \end{cases}
    \]
    and let $\{\mu_1,\ldots,\mu_k\}$ be the spectrum of~$M$.
    If $\mathbf{1}_{n_i}$ is a $\lambda_{i,n_i}$-eigenvector of~$L(G_i,\Sigma_i)$ for every $i\in[k]$, then the spectrum of $(G,\Sigma)$ is
    \[
        \{\mu_i: i \in [k]\}\cup\bigcup_{i=1}^k\{\lambda_{i,j}+m_i:j\in[n_i-1]\}.
    \]
\end{lemma}
\begin{proof}
    For each $i\in[k]$, let $L_i:=L(G_i,\Sigma_i)$.
    Note that $L(G,\Sigma)$ can be arranged as
    \[
        L(G,\Sigma)=\begin{bmatrix}
            L_1+m_1I_{n_1} & s_{1,2}J_{n_1\times n_2} & \cdots & s_{1,k}J_{n_1\times n_k}\\
            s_{2,1}J_{n_2\times n_1} & L_2+m_2I_{n_2} & \cdots & s_{2,k}J_{n_2\times n_k}\\
            \vdots & \vdots & \ddots & \vdots\\
            s_{k,1}J_{n_k\times n_1} & s_{k,2}J_{n_k\times n_2} & \cdots & L_k+m_kI_{n_k} 
        \end{bmatrix},
    \]
    where $s_{i,j}$ is~$1$ if $v_iv_j\in\Sigma'$, $-1$ if $v_iv_j\in E(G')-\Sigma'$, and $0$ otherwise.
    We first show that for each $i\in[k]$ and each $j\in[n_i-1]$, $\lambda_{i,j}+m_i$ is an eigenvalue of~$L(G,\Sigma)$.
    We choose a $\lambda_{i,j}$-eigenvector $x_{i,j}$ of~$L_i$ which is orthogonal to~$\mathbf{1}_{n_i}$, and let $y_{i,j}:=(\mathbf{0}^t,x_{i,j}^t,\mathbf{0}^t)^t$ where the first $\mathbf{0}^t$ has length $\sum_{\ell=1}^{i-1}n_\ell$.
    By the orthogonality,
    \begin{equation*}
        L(G,\Sigma)y_{i,j}
        =\begin{bmatrix}
            s_{1,i}J_{n_1\times n_i}x_{i,j}\\
            \vdots\\
            L_ix_{i,j}+m_ix_{i,j}\\
            \vdots\\
            s_{k,i}J_{n_k\times n_i}x_{i,j}
        \end{bmatrix}
        =\begin{bmatrix}
           \mathbf{0}_{n_1}\\
           \vdots\\
           \lambda_{i,j}x_{i,j}+m_ix_{i,j}\\
           \vdots\\
           \mathbf{0}_{n_k}
        \end{bmatrix}
        =(\lambda_{i,j}+m_i)y_{i,j}.
    \end{equation*}
    Hence, for each $i\in[k]$ and each $j\in[n_i-1]$, $\lambda_{i,j}+m_i$ is an eigenvalue of~$L(G,\Sigma)$.

    We now show that for each $i\in[k]$, $\mu_i$ is an eigenvalue of $L(G,\Sigma)$.
    Let $w_i:=(w_{i,1},\ldots,w_{i,k})^t$ be a $\mu_i$-eigenvector of~$M$, and $z_i:=(w_{i,1}\mathbf{1}_{n_1}^t,\ldots,w_{i,k}\mathbf{1}_{n_k}^t)^t$.
    Note that for each $j\in[k]$,
    \[
        \mu_iw_{i,j}=(\lambda_{j,n_j}+m_j)w_{i,j}+\sum_{\ell\in[k]-\{j\}}s_{j,\ell}\,n_\ell\,w_{i,\ell},
    \]
    and therefore
    \begin{align*}
        L(G,\Sigma)z_i
        &=\begin{bmatrix}
            (L_1\mathbf{1}_{n_1}+m_1\mathbf{1}_{n_1})w_{i,1}+s_{1,2}\,n_2\,w_{i,2}\mathbf{1}_{n_1}+\cdots+s_{1,k}\,n_k\,w_{i,k}\mathbf{1}_{n_1}\\
            s_{2,1}\,n_1\,w_{i,1}\mathbf{1}_{n_2}+(L_2\mathbf{1}_{n_2}+m_2\mathbf{1}_{n_2})w_{i,2}+\cdots+s_{2,k}\,n_k\,w_{i,k}\mathbf{1}_{n_2}\\
            \vdots\\
            s_{k,1}\,n_1\,w_{i,1}\mathbf{1}_{n_k}+s_{k,2}\,n_2\,w_{i,2}\mathbf{1}_{n_k}+\cdots+(L_k\mathbf{1}_{n_k}+m_k\mathbf{1}_{n_k})w_{i,k}
        \end{bmatrix} \\
        &=\begin{bmatrix}
            \left((\lambda_{1,n_1}+m_1)w_{i,1}+\sum_{\ell\in[k]-\{1\}}s_{1,\ell}\,n_\ell\,w_{i,\ell}\right)\mathbf{1}_{n_1}\\
            \left((\lambda_{2,n_2}+m_2)w_{i,2}+\sum_{\ell\in[k]-\{2\}}s_{2,\ell}\,n_\ell\,w_{i,\ell}\right)\mathbf{1}_{n_2}\\
            \vdots\\
            \left((\lambda_{k,n_k}+m_k)w_{i,k}+\sum_{\ell\in[k-1]}s_{k,\ell}\,n_\ell\,w_{i,\ell}\right)\mathbf{1}_{n_k}
        \end{bmatrix}
        =\begin{bmatrix}
            \mu_iw_{i,1}\mathbf{1}_{n_1}\\
            \mu_iw_{i,2}\mathbf{1}_{n_2}\\
            \vdots\\
            \mu_iw_{i,k}\mathbf{1}_{n_k}
        \end{bmatrix}
        =\mu_iz_i.
    \end{align*}
    Hence, for each $i\in[k]$, $\mu_i$ is an eigenvalue of $L(G,\Sigma)$, and this completes the proof.
\end{proof}

For every signed graph $(K_n,\Sigma)$, the following lemma finds an equivalent signed graph where every edge in~$X(\Sigma)$ is even.
This lemma will be used as a preprocessing step of the proofs in later subsections.

\begin{lemma}\label{lem:clean up}
    For every signed graph $(K_n,\Sigma)$ with $n\geq4$, there exists a signed graph $(K_k,\Sigma')$ with vertex set $\{v_1,\ldots,v_k\}$ such that $(K_n,\Sigma)$ is equivalent to a signed graph obtained from $(K_k,\Sigma')$ by substituting~$v_i$ with some signed complete graph $(H_i,\emptyset)$ for each $i\in[k]$ such that $X(\Sigma)=\bigcup_{i=1}^kE(H_i)$.
\end{lemma}
\begin{proof}
    Let $H_1,\ldots,H_k$ be the components of the graph on $V(K_n)$ with edge set $X(\Sigma)$.
    By~\Cref{lem:friendly}, each~$H_i$ is complete.
    For each $i\in[k]$, since $(H_i,\Sigma\cap E(H_i))$ has no odd triangles, it is equivalent to $(H_i,\emptyset)$.
    Therefore, there exists a signed graph $(K_n,\Gamma)$ equivalent to $(K_n,\Sigma)$ such that $X(\Sigma)\cap\Gamma=\emptyset$.

    We claim that for each distinct $i,j\in[k]$, the edges of $K_n$ between~$V(H_i)$ and~$V(H_j)$ have the same parity in~$(K_n,\Gamma)$.
    Let $vw$ and $v'w'$ be distinct edges of~$K_n$ between $V(H_i)$ and $V(H_j)$ with $\{v,v'\}\in V(H_i)$.
    Since $E(H_i)\cup E(H_j)$ is disjoint from~$\Gamma$, if $vw$ and $v'w'$ have a common end, then they have the same parity in~$(K_n,\Gamma)$.
    Thus, we may assume that they have no common end.
    Then each of~$vw$ and~$v'w'$ has the same parity with~$vw'$ in~$(K_n,\Gamma)$, so~$vw$ and~$v'w'$ have the same parity in~$(K_n,\Gamma)$.
    By the arbitrary choice of~$vw$ and~$v'w'$, this proves the claim.
    
    Let $v_1,\ldots,v_k$ be the vertices of~$K_k$, and let $\Sigma'$ be the set of edges $v_iv_j$ of~$K_k$ such that the edges of~$K_n$ between $V(H_i)$ and $V(H_j)$ are in~$\Gamma$.
    It is readily seen that $(K_n,\Gamma)$ can be obtained from $(K_k,\Sigma')$ by substituting each $v_i$ with $(H_i,\emptyset)$, and this completes the proof.
\end{proof}

\subsection{Spectral integral variation not captured by $X(\Sigma)$}\label{subsec:siv}

For a signed graph $(K_n,\Sigma)$ and each $X\subseteq X(\Sigma)$, the following lemma characterizes when we can add a new edge not in $X(\Sigma)$ to a signed graph $(K_n\setminus X,\Sigma-X)$ with spectral integral variation.

\begin{lemma}\label{lem:siv outside}
    For a signed graph $(K_n,\Sigma)$ with $n\geq4$, let $X\subseteq X(\Sigma)$ and $vw\in E(K_n)-X(\Sigma)$.
    Then spectral integral variation occurs under the addition of~$vw$ to~$(K_n\setminus(X\cup\{vw\}),\Sigma-(X\cup\{vw\}))$ with the same parity as in~$(K_n,\Sigma)$ if and only if $vw\in Y(\Sigma)$.
\end{lemma}
\begin{proof}
    Since each of $X(\Sigma)$ and $Y(\Sigma)$ remains unchanged after switching at vertices of~$K_n$, it suffices to show the statement for any signed graph equivalent to $(K_n,\Sigma)$.
    We start with switching at vertices of $K_n$ to find an equivalent signed graph fitting to our proof.
    By \Cref{lem:clean up}, we may assume that $(K_n,\Sigma)$ is obtained from a signed graph $(K_k,\Sigma')$ with vertex set $\{v_1,\ldots,v_k\}$ by substituting each~$v_i$ with a signed graph $(H_i,\emptyset)$ such that $X(\Sigma)=\bigcup_{i=1}^kE(H_i)$.
    For each $i\in[k]$, let $H'_i:=H_i\setminus(X\cap E(H_i))$.
    Note that $(K_n\setminus X,\Sigma-X)$ can be obtained from $(K_k,\Sigma')$ by substituting each~$v_i$ with $(H'_i,\emptyset)$.
    Since $vw\in E(K_n)-X(\Sigma)$, there are distinct $t,t'\in[k]$ with $v\in V(H_t)$ and $w\in V(H_{t'})$.
    Without loss of generality, we may assume that~$t=1$ and $t'=2$.
    For each $i\in[k]-\{1\}$ and distinct $u_1,u_2\in V(H_i)$, by the definition of substituting at~$v_1$, $u_1v$ and $u_2v$ have the same parity in~$(K_n,\Sigma)$.
    Similarly, for each $i'\in[k]-\{2\}$ and distinct $u'_1,u'_2\in V(H_{i'})$, $u'_1w$ and $u'_2w$ have the same parity in~$(K_n,\Sigma)$.
    Hence, by~\Cref{obs:centered}, we may further assume that $(K_n\setminus vw,\Sigma-\{vw\})$ is $(v,w)$-centered and $\Sigma\cap\bigcup_{i=1}^kE(H_i)=\emptyset$.
    Finally, by switching at~$w$ if necessary in the case of $\abs{V(H_1)}=\abs{V(H_2)}=1$, we assume that~$vw$ is an even edge of~$(K_n,\Sigma)$.
    Note that the edges of~$K_n$ between~$V(H_1)$ and~$V(H_2)$ are even edges in~$(K_n,\Sigma)$.
    We now arrange $L(K_n\setminus vw,\Sigma)$ as in~\eqref{eq:centered} and let $(G,\Gamma):=(K_n\setminus(X\cup\{vw\}),\Sigma-X)$.
    
    We first show the backward direction.
    Suppose that $vw\in Y(\Sigma)$.
    By~\Cref{obs:no common}, $V(H_1)=\{v\}$ and $V(H_2)=\{w\}$.
    For every $u\in V(K_n)-\{v,w\}$, since $uvw$ is odd and $(K_n\setminus vw,\Sigma)$ is $(v,w)$-centered, $u$ is an even neighbor of~$v$ and an odd neighbor of~$w$ in~$(K_n\setminus vw,\Sigma)$.
    In addition, by the definition of $Y(\Sigma)$, the number of odd neighbors of~$u$ in~$V(K_n)-\{v,w\}$ is equal to that of even neighbors of~$u$ in~$V(K_n)-\{v,w\}$.
    It implies that $L_{44}\mathbf{1}=(n-1)\mathbf{1}$, thus by~\Cref{thm:main2 even}, spectral integral variation occurs under the addition of an even edge~$vw$ to $(K_n\setminus vw,\Sigma)$.

    We show that spectral integral variation also occurs under the addition of an even edge~$vw$ to~$(G,\Gamma)$.
    Since $V(H_1)=\{v\}$ and $V(H_2)=\{w\}$, $(G,\Gamma)$ can be obtained from $(K_k\setminus v_1v_2,\Sigma')$ by substituting each~$v_i$ with $(H'_i,\emptyset)$.
    By~\Cref{lem:substitution}, spectral integral variation occurs under the addition of an even edge~$vw$ to~$(K_n\setminus vw,\Sigma)$ if and only if the spectra of $M(K_k,\Sigma')$ and $M(K_k\setminus v_1v_2,\Sigma')$ differ by integer quantities if and only if spectral integral variation occurs under the addition of an even edge~$vw$ to~$(G,\Gamma)$.
    Hence, spectral integral variation occurs under the addition of an even edge~$vw$ to $(G,\Gamma)$.

    We now show the forward direction.
    Suppose that spectral integral variation occurs under the addition of an even edge~$vw$ to~$(G,\Gamma)$.
    Since $v\in V(H_1)$ and $w\in V(H_2)$, every vertex in $V(G)-\{v,w\}$ is adjacent in~$(G,\Gamma)$ to at least one of~$v$ and~$w$.
    Let $A:=N_G(v)-N_G(w)$, $B:=N_G(w)-N_G(v)$, $C:=N_{G,\Gamma}^+(v)\cap N_{G,\Gamma}^+(w)$, and $D:=N_{G,\Gamma}^+(v)\cap N_{G,\Gamma}^-(w)$.
    Since $(G,\Gamma)$ is $(v,w)$-centered, $(A,B,C,D)$ is a partition of $V(G)-\{v,w\}$.
    Since $H_1$ and $H_2$ are disjoint, we have that $\abs{D}\geq1$.
    Note that
    \begin{align*}
        \{v\}\cup B\subseteq V(H_1)&\subseteq\{v\}\cup B\cup C,\\
        \{w\}\cup A\subseteq V(H_2)&\subseteq\{w\}\cup A\cup C,
    \end{align*}
    and for each $i\in[k]-[2]$, $V(H_i)$ is a subset of~$C$ or~$D$.
    Hence, the degrees of~$v$ and~$w$ in~$G$ are $d_1=n-2-\abs{B}$ and $d_2=n-2-\abs{A}$, respectively.
    In addition, each vertex in~$C$ is adjacent in~$G$ to every vertex in~$D$, and vice versa.

    \begin{claim}\label{clm:singleton}
        $\abs{V(H_1)}=\abs{V(H_2)}=1$.
    \end{claim}
    \begin{subproof}
        We first show that $(V(H_1)\cup V(H_2))\cap C=\emptyset$.
        Suppose for contradiction that $(V(H_1)\cup V(H_2))\cap C$ contains~$u$.
        If $u\in V(H_2)$, then for the number $d(u,A)$ of edges between~$u$ and~$A$, by~\Cref{thm:main2 even}\eqref{eq:main 3},
        \[
            -d(u,A)+\abs{B}+2\abs{D}=\abs{B}-\abs{A},
        \]
        and therefore $\abs{A}\geq d(u,A)=\abs{A}+2\abs{D}\geq\abs{A}+2$, a contradiction.
        Otherwise, for the number $d(u,B)$ of edges between~$u$ and~$B$, by~\Cref{thm:main2 even}\eqref{eq:main 3},
        \[
            -\abs{A}+d(u,B)-2\abs{D}=\abs{B}-\abs{A},
        \]
        and therefore $\abs{B}\geq d(u,B)=\abs{B}+2\abs{D}\geq\abs{B}+2$, a contradiction.
        Hence, $(V(H_1)\cup V(H_2))\cap C=\emptyset$.
        
        Since $V(H_1)=\{v\}\cup B$ and $V(H_2)=\{w\}\cup A$, it suffices to show that $A\cup B=\emptyset$.
        Suppose for contradiction that $A\neq\emptyset$.
        Note that each vertex in~$A$ is adjacent in~$G$ to every vertex in~$V(G)-(A\cup\{w\})$.
        Since $\Sigma\cap E(H_2)=\emptyset$, $L_{11}\mathbf{1}=(n-1-\abs{A})\mathbf{1}$, and therefore by~\Cref{thm:main2 even}\eqref{eq:main 1},
        \[
            (n-1-\abs{A})+\abs{B}+2\abs{D}=n-1-\abs{A},
        \]
        contradicting that $\abs{D}\geq1$.
        Thus, $A=\emptyset$.

        We may assume that $B\neq\emptyset$, because otherwise we are done.
        Note that each vertex in~$B$ is adjacent in~$G$ to every vertex in $V(G)-(B\cup\{v\})$.
        Since $\Sigma\cap E(H_1)=\emptyset$, $L_{22}\mathbf{1}=(n-1-\abs{B})\mathbf{1}$, and therefore by~\Cref{thm:main2 even}\eqref{eq:main 2},
        \[
            -\abs{A}-(n-1-\abs{B})-2\abs{D}=-(n-1-\abs{B})
        \]
        contradicting that $\abs{D}\geq1$, and this proves the claim.
    \end{subproof}

    By~\Cref{clm:singleton}, $V(G)-\{v,w\}=C\cup D$, so by the assumption on~$(K_n,\Sigma)$, $vw$ is not contained in~$\Sigma$.
    For every $u\in V(G)-\{v,w\}$, let $d^+(u,D):=\abs{D\cap N_{G,\Gamma}^+(u)}$, $d^-(u,D):=\abs{D\cap N_{G,\Gamma}^-(u)}$, and $t_u:=\abs{D\setminus N_G(u)}$.
    Note that $\abs{D}=t_u+d^+(u,D)+d^-(u,D)$ for every $u\in V(G)-\{v,w\}$.

    For every $u\in D$, by~\Cref{thm:main2 even}\eqref{eq:main 4}, $2(n-t_u-d^+(u,D)+d^-(u,D))=2(n-1)$, and therefore
    \begin{equation}
        d^-(u,D)=t_u+d^+(u,D)-1.\label{eq:odd}
    \end{equation}
    Thus, $\abs{D}=t_u+d^+(u,D)+d^-(u,D)=2d^-(u,D)+1$, which is odd.
    If~$C$ contains some vertex~$u'$, then~$u'$ is adjacent to every vertex in~$D$, so by~\Cref{thm:main2 even}\eqref{eq:main 3}, $-d^+(u',D)+d^-(u',D)=0$, contradicting that~$\abs{D}$ is odd.
    Hence, $C=\emptyset$.

    We now show that $vw\in Y(\Sigma)$.
    Since $D=V(K_n)-\{v,w\}$ and $vw\notin\Sigma$, for every $u\in D$, $uvw$ is odd in~$(K_n,\Sigma)$.
    Since $(K_n\setminus vw,\Sigma)$ is $(v,w)$-centered, $v$ is incident with no odd edge of~$(K_n,\Sigma)$.
    Thus, for each $u\in D$, the number of odd (resp. even) triangles of~$(G,\Gamma)$ containing~$u$ and~$v$ is $d^-(u,D)$ (resp. $d^+(u,D)$).
    On the other hand, since $X(\Sigma)=\bigcup_{i=1}^kE(H_i)$ and $\Sigma\cap\bigcup_{i=1}^kE(H_i)=\emptyset$, for each $u\in D$, $uvw$ is the unique odd triangle of~$(K_n,\Sigma)$ containing~$u,v$ and an edge in~$E(K_n)-E(G)$, and the number of even triangles of~$(K_n,\Sigma)$ containing~$u,v$ and an edge in $E(K_n)-E(G)$ is~$t_u-1$.
    Thus, for each $u\in D$, by~\eqref{eq:odd}, the number of odd triangles of~$(K_n,\Sigma)$ containing~$u$ and~$v$ is one more than the number of even triangles of~$(K_n,\Sigma)$ containing~$u$ and~$v$, which implies that $vw\in Y(\Sigma)$.
\end{proof}

\subsection{Proof of~\Cref{thm:main3}}\label{subsec:main proof}

We will use the following lemma to prove~\Cref{thm:main3}.

\begin{lemma}\label{lem:poset}
    For a signed graph $(K_n,\Sigma)$, let $\mathcal{C}$ be a set of $n$-vertex signed graphs satisfying the following:
    \begin{enumerate}[label=(\roman*)]
        \item\label{equiv1} $(K_n,\Sigma')\in\mathcal{C}$ if and only if $\Sigma'=\Sigma$,
        \item\label{equiv2} for each $(G,\Sigma')\in\mathcal{C}$ with $G\neq K_n$, there exists an edge $vw\in E(K_n)-E(G)$ such that either $(G+vw,\Sigma')$ or $(G+vw,\Sigma'\cup\{vw\})$ is in~$\mathcal{C}$, and
        \item\label{equiv3} for each $(G,\Sigma')\in\mathcal{C}$ and $vw\in E(G)$, $(G\setminus vw,\Sigma'-\{vw\})\in\mathcal{C}$ if and only if spectral integral variation occurs under the addition of~$vw$ to $(G\setminus vw,\Sigma'-\{vw\})$ with the same parity as in~$(K_n,\Sigma)$.
    \end{enumerate}
    Then a signed graph is integrally $\Sigma$-completable if and only if it is contained in~$\mathcal{C}$.
\end{lemma}
\begin{proof}
    Let $(H,\Gamma)$ be a signed graph.
    We proceed by induction on $k:=\abs{E(K_n)-E(H)}$ to show that $(H,\Gamma)$ is integrally $\Sigma$-completable if and only if $(H,\Gamma)\in\mathcal{C}$.
    The base case is obvious from the definition of integrally $\Sigma$-completability and~\ref{equiv1}.
    Thus, we may assume that $k>0$.

    We first show the backward direction.
    Suppose that $(H,\Gamma)\in\mathcal{C}$.
    Since $k\geq1$, by~\ref{equiv2}, there exists an edge $vw\in E(K_n)-E(H)$ such that $(H+vw,\Gamma')\in\mathcal{C}$ for some $\Gamma'\in\{\Gamma,\Gamma\cup\{vw\}\}$.
    By the inductive hypothesis, $(H+vw,\Gamma')$ is integrally $\Sigma$-completable.
    Then by~\ref{equiv3}, spectral integral variation occurs under the addition of $vw$ to $(H,\Gamma)$ with the same parity as in $(H+vw,\Gamma')$.
    Hence, $(H,\Gamma)$ is integrally $\Sigma$-completable.

    We now show the forward direction.
    Suppose that $(H,\Gamma)$ is integrally $\Sigma$-completable.
    Since $k\geq1$, there exists an edge $vw\in E(K_n)-E(H)$ such that spectral integral variation occurs under the addition of~$vw$ to $(H,\Gamma)$ with the same parity as in $(K_n,\Sigma)$.
    Note that the resulting signed graph is also integrally $\Sigma$-completable, so by the inductive hypothesis, it is contained in~$\mathcal{C}$.
    Then by~\ref{equiv3}, $(H,\Gamma)\in\mathcal{C}$.

    This completes the proof by induction.
\end{proof}

We now prove~\Cref{thm:main3}.

\begin{proof}[Proof of~\Cref{thm:main3}]
    Let $\mathcal{C}$ be the set of $n$-vertex signed graphs $(H,\Gamma)$ such that $E(K_n)-E(H)\subseteq X(\Sigma)\cup Y(\Sigma)$, $\Gamma=\Sigma\cap E(H)$, and $K_n\setminus(X(\Sigma)-E(H))$ is integrally completable.
    By~\Cref{lem:poset}, it suffices to show that~$\mathcal{C}$ satisfies~\Cref{lem:poset}\ref{equiv1}--\ref{equiv3}.
    It is readily seen that~$\mathcal{C}$ satisfies~\Cref{lem:poset}\ref{equiv1}.

    We show that~$\mathcal{C}$ satisfies~\Cref{lem:poset}\ref{equiv2}.
    Let $(H,\Gamma)$ be a signed graph in~$\mathcal{C}$ with $H\neq K_n$ and let $Z:=E(K_n)-E(H)$.
    Recall that $Z\subseteq X(\Sigma)\cup Y(\Sigma)$.
    Suppose first that $Z\cap X(\Sigma)=\emptyset$.
    Then $Z=Y(\Sigma)$ which is a singleton by~\Cref{lem:suspicious}.
    Since $(K_n,\Sigma)\in\mathcal{C}$, $\mathcal{C}$ satisfies~\Cref{lem:poset}\ref{equiv2} in this case.
    Thus, we may assume that $Z\cap X(\Sigma)\neq\emptyset$.
    Since $K_n\setminus(X(\Sigma)-E(H))=K_n\setminus(X(\Sigma)\cap Z)$ is integrally completable, we can choose an edge $vw\in X(\Sigma)\cap Z$ such that~$K_n\setminus(X(\Sigma)-E(H+vw))$ is integrally completable.
    If $vw\in\Sigma$, then $(H+vw,\Gamma\cup\{vw\})$ is in~$\mathcal{C}$, and otherwise $(H+vw,\Gamma)\in\mathcal{C}$.
    Hence, $\mathcal{C}$ satisfies~\Cref{lem:poset}\ref{equiv2}.

    We now show that~$\mathcal{C}$ satisfies~\Cref{lem:poset}\ref{equiv3}.
    Since spectral integral variation is invariant under switching at vertices of~$K_n$, by~\Cref{lem:clean up}, we may assume that $(K_n,\Sigma)$ is obtained from a signed graph~$(K_k,\Sigma')$ with vertex set~$\{v_1,\ldots,v_k\}$ by substituting~$v_i$ with some signed complete graph $(H_i,\emptyset)$ for each $i\in[k]$ such that $X(\Sigma)=\bigcup_{i=1}^kE(H_i)$.
    Let $(H,\Gamma)$ be a signed graph in~$\mathcal{C}$, let $Z:=E(K_n)-E(H)$, and let $vw\in E(H)$.
    For each $i\in[k]$, let $H'_i:=H_i\setminus(Z\cap E(H_i))$.
    
    We consider four cases depending on whether $vw\in X(\Sigma)$ and whether $Z\cap Y(\Sigma)=\emptyset$.

    \medskip
    \noindent\textbf{Case 1.} $\{vw\}\cup Z\subseteq X(\Sigma)$.

    Since $K_n\setminus(X(\Sigma)-E(H))=K_n\setminus Z=H$, $H$ is integrally completable by the definition of~$\mathcal{C}$, and $(H,\Gamma)$ can be obtained from $(K_k,\Sigma')$ by substituting each $v_i$ with $(H'_i,\emptyset)$.
    Recall that a graph is integrally completable if and only if it has no induced subgraph isomorphic to $P_4$ or $2K_2$~\cite{kirkland05}.
    For each $i\in[k]$, since each vertex of $H'_i$ is adjacent in~$H$ to every vertex in~$V(K_n)-V(H'_i)$ and both $P_4$ and $2K_2$ have maximum degree at most~$2$, if $H$ has~$P_4$ or~$2K_2$ as an induced subgraph, then it is an induced subgraph of $H_\ell$ for some $\ell\in[k]$.
    Thus, we deduce that~$H$ is integrally completable if and only if $H'_i$ is integrally completable for every $i\in[k]$.
    Since $vw\in X(\Sigma)$, there exists $j\in[k]$ with $vw\in E(H_j)$.
    By~\Cref{lem:substitution}, the spectral variation occurred by the addition of an even edge~$vw$ to~$(H\setminus vw,\Gamma)$ is the same as that occurred by the addition of an edge~$vw$ to $H'_j\setminus vw$.
    Thus, the following hold.
    \begin{enumerate}
        \item[] Spectral integral variation occurs under the addition of an even edge~$vw$ to~$(H\setminus vw,\Gamma)$.
        \item[$\Leftrightarrow$] Spectral integral variation occurs under the addition of an edge~$vw$ to~$H'_j\setminus vw$.
        \item[$\Leftrightarrow$] Spectral integral variation occurs under the addition of an edge~$vw$ to~$H\setminus vw$.
        \item[$\Leftrightarrow$] $H\setminus vw=K_n\setminus(X(\Sigma)-E(H\setminus vw))$ is integrally completable.
        \item[$\Leftrightarrow$] $(H\setminus vw,\Gamma)\in\mathcal{C}$.
    \end{enumerate}
    Hence, $\mathcal{C}$ satisfies~\Cref{lem:poset}\ref{equiv3}.
    
    \medskip
    \noindent\textbf{Case 2.} $vw\in X(\Sigma)$ and $Z\cap Y(\Sigma)\neq\emptyset$.

    Let $v'w'$ be an edge in $Z\cap Y(\Sigma)$, which is unique by~\Cref{lem:suspicious}.
    Since $v'w'\in Y(\Sigma)$, there are distinct $\ell,\ell'\in[k]$ with $V(H_\ell)=\{v'\}$ and $V(H_{\ell'})=\{w'\}$.
    Since $vw\in X(\Sigma)$, there exists $j\in[k]$ with $vw\in E(H_j)$.
    By~\Cref{obs:no common}, $j$, $\ell$, and $\ell'$ are pairwise distinct.
    We may assume that $\ell=1$ and $\ell=2$.
    
    Since $K_n\setminus(X(\Sigma)-E(H))=K_n\setminus(Z-\{v'w'\})=H+v'w'$, $(H,\Gamma)$ can be obtained from $(K_k\setminus v_1v_2,\Sigma'-\{v_1v_2\})$ by substituting each~$v_i$ with~$(H'_i,\emptyset)$.
    For each $i\in[k]$, since each vertex of $H'_i$ is adjacent in~$H+v'w'$ to every vertex in~$V(K_n)-V(H'_i)$, $H+v'w'$ is integrally completable if and only if $H'_i$ is integrally completable for every $i\in[k]$.
    By~\Cref{lem:substitution}, the spectral variation occurred by the addition of an even edge~$vw$ to~$(H\setminus vw,\Gamma)$ is the same as that occurred by the addition of an edge~$vw$ to~$H'_j\setminus vw$.
    Let $H':=(H+v'w')\setminus vw$ and let $\Gamma'$ be $\Gamma\cup\{v'w'\}$ if $v'w'\in\Sigma$, and $\Gamma$ otherwise.
    Thus, the following hold.
    \begin{enumerate}
        \item[] Spectral integral variation occurs under the addition of an even edge~$vw$ to~$(H\setminus vw,\Gamma)$.
        \item[$\Leftrightarrow$] Spectral integral variation occurs under the addition of an edge~$vw$ to~$H'_j\setminus vw$.
        \item[$\Leftrightarrow$] Spectral integral variation occurs under the addition of an edge~$vw$ to~$H'$.
        \item[$\Leftrightarrow$] $(H',\Gamma')\in\mathcal{C}$.
        \item[$\Leftrightarrow$] $K_n\setminus(X(\Sigma)-E(H'))=K_n\setminus(X(\Sigma)-E(H\setminus vw))$ is integrally completable.
        \item[$\Leftrightarrow$] $(H\setminus vw,\Gamma)\in\mathcal{C}$.
    \end{enumerate}
    Here, the third and the fourth are equivalent by the previous case.
    Hence, $\mathcal{C}$ satisfies~\Cref{lem:poset}\ref{equiv3}.

    \medskip
    \noindent\textbf{Case 3.} $vw\notin X(\Sigma)$ and $Z\subseteq X(\Sigma)$.

    For the backward direction of~\Cref{lem:poset}\ref{equiv3}, suppose first that spectral integral variation occurs under the addition of~$vw$ to $(H\setminus vw,\Gamma-\{vw\})$ with the same parity as in~$(K_n,\Sigma)$.
    Since $vw\notin X(\Sigma)$, by~\Cref{lem:siv outside}, $vw\in Y(\Sigma)$.
    Then $E(K_n)-E(H\setminus vw)=Z\cup\{vw\}\subseteq X(\Sigma)\cup Y(\Sigma)$ and $K_n\setminus(X(\Sigma)-E(H))=K_n\setminus(X(\Sigma)-E(H\setminus vw))$.
    Hence, $(H\setminus vw,\Gamma-\{vw\})\in\mathcal{C}$.

    For the forward direction of~\Cref{lem:poset}\ref{equiv3}, suppose that $(H\setminus vw,\Gamma-\{vw\})\in\mathcal{C}$.
    By the definition of~$\mathcal{C}$, $Z\cup\{vw\}\subseteq X(\Sigma)\cup Y(\Sigma)$.
    Since $vw\notin X(\Sigma)$, $vw$ is contained in~$Y(\Sigma)$.
    Then by~\Cref{lem:siv outside}, spectral integral variation occurs under the addition of~$vw$ to $(K_n\setminus(Z\cup\{vw\}),\Sigma-(Z\cup\{vw\}))=(H\setminus vw,\Gamma-\{vw\})$ with the same parity as in~$(K_n,\Sigma)$.

    \medskip
    \noindent\textbf{Case 4.} $vw\notin X(\Sigma)$ and $Z\cap Y(\Sigma)\neq\emptyset$.

    Let $v'w'$ be an edge in $Z\cap Y(\Sigma)$, which is unique by~\Cref{lem:suspicious}.
    As $vw\in E(H)$ and $v'w'\notin E(H)$, they are distinct, so $vw\notin Y(\Sigma)$, that is, $vw\notin X(\Sigma)\cup Y(\Sigma)$.
    By the definition of~$\mathcal{C}$, $(H\setminus vw,\Gamma-\{vw\})\notin\mathcal{C}$.
    Hence, to show that~$\mathcal{C}$ satisfies~\Cref{lem:poset}\ref{equiv3}, it suffices to show that spectral integral variation does not occur under the addition of~$vw$ to~$(H\setminus vw,\Gamma-\{vw\})$ with the same parity as in~$(K_n,\Sigma)$.
    Suppose not.
    To derive a contradiction, we consider $(K_n,\Sigma\triangle\{v'w'\})$.
    Since $v'w'\in Y(\Sigma)$, by~\Cref{lem:swapping}, $Z\subseteq X(\Sigma)\cup\{v'w'\}=X(\Sigma\triangle\{v'w'\})$ and $Y(\Sigma\triangle\{v'w'\})=\emptyset$.
    Since $vw\notin X(\Sigma)$, $vw$ is not contained in~$X(\Sigma\triangle\{v'w'\})$.
    Since spectral integral variation occurs under the addition of~$vw$ to~$(H\setminus vw,\Gamma-\{vw\})=(K_n\setminus(Z\cup\{vw\}),\Sigma-(Z\cup\{vw\}))$ with the same parity as in~$(K_n,\Sigma\triangle\{v'w'\})$, by~\Cref{lem:siv outside}, $vw\in Y(\Sigma\triangle\{v'w'\})$, contradicting that $Y(\Sigma\triangle\{v'w'\})=\emptyset$.
    This completes the proof.
\end{proof}

\bibliographystyle{amsplain}

\providecommand{\bysame}{\leavevmode\hbox to3em{\hrulefill}\thinspace}
\providecommand{\MR}{\relax\ifhmode\unskip\space\fi MR }
\providecommand{\MRhref}[2]{%
  \href{http://www.ams.org/mathscinet-getitem?mr=#1}{#2}
}
\providecommand{\href}[2]{#2}

\end{document}